\documentclass[11pt,fleqn]{article}

\usepackage{amsfonts}
\usepackage{amsmath}
\usepackage{amssymb}
\usepackage{amsthm}
\usepackage{bbm}
\usepackage{bm}
\usepackage{booktabs}
\usepackage{cases}
\usepackage{cite}
\usepackage{graphicx}
\usepackage{indentfirst}
\usepackage{longtable}
\usepackage{makecell}
\usepackage{mathrsfs}
\usepackage{titlesec}
\usepackage{rotating}
\usepackage[T1]{fontenc}
\usepackage[centerlast]{caption}

\allowdisplaybreaks

\numberwithin{equation}{section}

\newtheorem{theorem}{Theorem}[section]
\newtheorem{lemma}[theorem]{Lemma}

\newtheorem{corollary}[theorem]{Corollary}

\theoremstyle{definition}

\newtheorem{remark}[theorem]{Remark}
\newtheorem{example}{Example}[section]

\titleformat{\section}
{\normalfont\normalsize\bf}{\thesection.}{6pt}{}
\titleformat{\subsection}
{\normalfont\normalsize}{\thesubsection.}{6pt}{}
\titleformat{\subsubsection}
{\normalfont\normalsize}{\thesubsubsection.}{6pt}{}

\newcommand{\bibb}[1]{\left\{#1\right\}}

\newcommand{\smbb}[1]{\left(#1\right)}

\newcommand{\tf}{\tfrac}

\newcommand{\al}{\alpha}
\newcommand{\be}{\beta}
\newcommand{\ga}{\gamma}
\newcommand{\de}{\delta}

\newcommand{\la}{\lambda}
\newcommand{\si}{\sigma}
\newcommand{\vep}{\varepsilon}
\newcommand{\ze}{\zeta}

\newcommand{\vOm}{\varOmega}

\newcommand{\vPsi}{\varPsi}

\newcommand{\ue}{\mathrm{e}}

\newcommand{\Lii}{\,\mathrm{Li}}

\newcommand{\ds}{\displaystyle}

\newcommand{\ol}{\overline}

\def\t{\tilde{t}}
\def\S{\tilde{S}}

\makeatletter
\newcommand{\tmod}[1]{{\@displayfalse\pmod{#1}}}
\makeatother

\makeatletter

\newdimen\bibspace
\renewenvironment{thebibliography}[1]{%
 \section*{\refname 
       \@mkboth{\MakeUppercase\refname}{\MakeUppercase\refname}}%
     \list{\@biblabel{\@arabic\c@enumiv}}%
          {\settowidth\labelwidth{\@biblabel{#1}}%
           \leftmargin\labelwidth
           \advance\leftmargin\labelsep
           \itemsep\bibspace
           \parsep\z@skip     %
           \@openbib@code
           \usecounter{enumiv}%
           \let\p@enumiv\@empty
           \renewcommand\theenumiv{\@arabic\c@enumiv}}%
     \sloppy\clubpenalty4000\widowpenalty4000%
     \sfcode`\.\@m}
    {\def\@noitemerr
      {\@latex@warning{Empty `thebibliography' environment}}%
     \endlist}

\makeatother

\begin{document}

\title{\bf\boldmath{On variants of the Euler sums and symmetric extensions of the Kaneko-Tsumura conjecture}}
\author
{
Weiping Wang$\,^{a,}$\thanks{\emph{E-mail addresses\,:}
wpingwang@yahoo.com, wpingwang@zstu.edu.cn (Weiping Wang).}
\quad
Ce Xu$\,^{b,}$\thanks{Corresponding author. \emph{E-mail address\,:}
cexu2020@ahnu.edu.cn (Ce Xu).}
\\[1pt]
\small $a.$ School of Science, Zhejiang Sci-Tech University,
    Hangzhou 310018, P.R. China\\
\small $b.$ School of Mathematics and Statistics, Anhui Normal University,
    Wuhu 241002, P.R. China
}
\date{}
\maketitle

\vspace{-0.5cm}
\begin{center}
\parbox{6.3in}{\small{\bf Abstract}\vspace{3pt}

\hspace{3.5ex}
By using various expansions of the parametric digamma function and the method of residue computations, we study three variants of the linear Euler sums, related Hoffman's double $t$-values and Kaneko-Tsumura's double $T$-values, and establish several symmetric extensions of the Kaneko-Tsumura conjecture. Some special cases are discussed in detail to determine the coefficients of involved mathematical constants in the evaluations. In particular, it can be found that several general convolution identities on the classical Bernoulli numbers and Genocchi numbers are required in this study, and they are verified by the derivative polynomials of hyperbolic tangent.
}

\vspace{6pt}
\parbox{6.3in}{\small{\emph{AMS classification\,:}}\,\,
11M32; 11A07; 40A25; 05A19; 11B68}

\vspace{4.5pt}
\parbox{6.3in}{\small{\emph{Keywords\,:}}\,\,
Multiple zeta values; Multiple $t$-values; Multiple $T$-values; Harmonic numbers; Euler sums; Convolution identities; Bernoulli numbers; Genocchi numbers}
\end{center}


\setcounter{tocdepth}{2}
\tableofcontents


\section{Introduction}\label{Sec.intro}

The \emph{alternating multiple zeta values} (alternating MZVs) are defined by
\begin{equation}\label{AMZV}
\ze(s_1,s_2,\ldots,s_k;\si_1,\si_2,\ldots,\si_k)
    =\sum_{n_1>n_2>\cdots>n_k\geq 1}\frac{\si_1^{n_1}\si_2^{n_2}\cdots\si_k^{n_k}}
        {n_1^{s_1}n_2^{s_2}\cdots n_k^{s_k}}\,,
\end{equation}
where $s_j$ are positive integers, $\si_j=\pm1$, for $j=1,2,\ldots,k$, with $(s_1,\si_1)\neq(1,1)$. As usual, we can combine the strings of exponents and signs into a single string, with $s_j$ in the $j$th position when $\si_j=+1$, and $\bar{s}_j$ in the $j$th position when $\si_j=-1$. In particular, when $s_1>1$, setting $\si_j=1$, $j=1,2,\ldots,k$, in Eq. (\ref{AMZV}), we obtain \emph{multiple zeta values} (MZVs), and denote them by $\ze(s_1,s_2,\ldots,s_k)$.

The study of multiple zeta values began in the early 1990s with the works of Hoffman \cite{Hoff92} and Zagier \cite{Zag92}, and has attracted a lot of research in the last two decades. For detailed history and applications, the readers may consult in, e.g., the book of Zhao \cite{Zhao16.MZF}.

In a recent paper \cite{Hoff19.AOV}, Hoffman introduced and studied an ``odd'' variant of the MZVs:
\begin{align*}
t(s_1,s_2,\ldots,s_k)
&=\sum_{\substack{n_1>n_2>\cdots>n_k\geq 1\\n_i\ \mathrm{odd}}}
    \frac{1}{n_1^{s_1}n_2^{s_2}\cdots n_k^{s_k}}\\
&=\sum_{n_1>n_2>\cdots>n_k\geq 1}
    \frac{1}{(2n_1-1)^{s_1}(2n_2-1)^{s_2}\cdots(2n_k-1)^{s_k}}\,,
\end{align*}
which are called the \emph{multiple $t$-values} (MtVs). For convenience, let us call the normalized version
\begin{align*}
\t(s_1,s_2,\ldots,s_k)
&=\sum_{n_1>n_2>\cdots>n_k\geq 1}
    \frac{1}{(n_1-1/2)^{s_1}(n_2-1/2)^{s_2}\cdots(n_k-1/2)^{s_k}}\\
&=2^{s_1+s_2+\cdots+s_k}t(s_1,s_2,\ldots,s_k)
\end{align*}
the \emph{multiple $\t$-values}. According to the definitions, $\t(s)=2^st(s)=(2^s-1)\ze(s)$ for integer $s\geq 2$, where $\ze(s)$ is the \emph{Riemann zeta function}. As showed in \cite[Corollaries 4.1 and 4.2]{Hoff19.AOV}, the MtVs are reducible to linear combinations of alternating MZVs. Due to this fact as well as the congruence condition in the summation, the MtVs can be regarded as MZVs of level two. In 2020, Kaneko and Tsumura \cite{KaTs20} introduced another variant of MZVs of level two:
\begin{align*}
T(s_1,s_2,\ldots,s_k)
&=2^k\sum_{\substack{n_1>n_2>\cdots>n_k\geq1\\n_i\equiv k-i+1\tmod{2}}}
    \frac{1}{n_1^{s_1}n_2^{s_2}\cdots n_k^{s_k}}\nonumber\\
&=2^k\sum_{n_1>n_2>\cdots>n_k\geq 1}
    \frac{1}{(2n_1-k)^{s_1}(2n_2-k+1)^{s_2}\cdots(2n_k-1)^{s_k}}\,,
\end{align*}
which are called the \emph{multiple $T$-values} (MTVs).

By convention, for the MZVs, MtVs and MTVs, the quantity $k$ is called the ``\emph{depth}'' and the quantity $w:=s_1+s_2+\cdots+s_k$ is called the ``\emph{weight}''.

Let $\mathcal{Z}$ be the space of usual multiple zeta values. In \cite[Conjecture 5.3]{KaTs20}, Kaneko and Tsumura observed that the following relation holds:
\begin{equation}\label{KT.conj}
\sum_{\substack{i+j=m\\i,j\geq 0}}
    \binom{p+i-1}{i}\binom{q+j-1}{j}T(p+i,q+j)\in\mathcal{Z}\,,
\end{equation}
for $m,q\geq 1$ and $p\geq 2$, with $m+p+q$ even. That is, these sums are expressible in terms of MZVs. In 2021, Murakami \cite[Theorem 42]{Mura21} proved this conjecture by using the motivic method employed in \cite{Gla18} (see also \cite[Remark 5.6]{KaTs20}).

For simplicity, let us define the transformation operator $\la_m$ by
\begin{equation}
\la_m(\vOm_{p,q})
    :=\sum_{\substack{i+j=m-1\\i,j\geq 0}}
    \binom{p+i-1}{i}\binom{q+j-1}{j}\vOm_{p+i,q+j}\,,\quad\text{for }m\geq 1\,,
\end{equation}
which maps a sum $\vOm_{p,q}$ of two parameters to another one of three parameters, with $\la_1(\vOm_{p,q})=\vOm_{p,q}$. Then Kaneko and Tsumura's conjecture (\ref{KT.conj}) can be rewritten as
\[
\la_{m+1}(T(p,q))\in\mathcal{Z}\,,
    \quad\text{for } m,q\geq 1,\ p\geq 2,\ \text{with } m+p+q \text{ even}.
\]

Let $H_n^{(r)}$ and $h_n^{(r)}$ stand for the \emph{generalized harmonic numbers} and the \emph{odd harmonic numbers} of order $r$, respectively, defined by $H_0^{(r)}=h_0^{(r)}=0$ and
\[
H_n^{(r)}=\sum_{k=1}^n\frac{1}{k^r}\,,\quad
    h_n^{(r)}=\sum_{k=1}^n\frac{1}{(k-1/2)^r}\,,\quad\text{for } n,r=1,2,\ldots,
\]
with $H_n\equiv H_n^{(1)}$ and $h_n\equiv h_n^{(1)}$. The classical \emph{Euler sums} are infinite series
\[
S_{p_1p_2\cdots p_k,q}
    =\sum_{n=1}^\infty\frac{H_n^{(p_1)}H_n^{(p_2)}\cdots H_n^{(p_k)}}{n^q}\,,
\]
where $p_1\leq p_2\leq\ldots\leq p_k$ and $q\geq 2$, and the linear sums are of the form $S_{p,q}$. For an early introduction and study on the evaluations of the classical Euler sums, the readers may consult in Flajolet and Salvy's paper \cite{FlSa98}, and for some recent progress, the readers are referred to \cite{WangLyu18,Xu17.MZVES,XuWang19.EFES} and references therein.

In this paper, using various expansions of the parametric digamma function and the method of residue computations, we establish symmetric extensions of the Kaneko-Tsumura conjecture (\ref{KT.conj}) on three variants of the linear Euler sums, defined by
\[
T_{p,q}:=\sum_{n=1}^{\infty}\frac{h_{n-1}^{(p)}}{(n-1/2)^q}\,,\quad
\S_{p,q}:=\sum_{n=1}^{\infty}\frac{h_n^{(p)}}{n^q}\,,\quad
R_{p,q}:=\sum_{n=1}^{\infty}\frac{H_{n-1}^{(p)}}{(n-1/2)^q}\,,
    \quad\text{for }q\geq 2\,,
\]
respectively, which were introduced and studied in \cite{WangXu20.DTE,Xu20.EFS}. In particular, we show that for $m,p\geq 1$ and $q\geq2$, there hold
\begin{align}
&(-1)^{m-1}\la_p(T_{m,q})+(-1)^{p-1}\la_m(\S_{p,q})
    \in\mathbb{Q}[\ln(2),\text{zeta values}]\,,\label{intro.sym.TS}\\
&(-1)^{m-1}\la_p(R_{m,q})+(-1)^{p-1}\la_m(R_{p,q})
    \in\mathbb{Q}[\ln(2),\text{zeta values}]\,.\label{intro.sym.RR}
\end{align}
In other words, both of the symmetric sums on the left are reducible to $\ln(2)$ and zeta values. Explicit expressions of these two sums are presented, and a detailed discussion on some interesting special cases follows.

Moreover, by the derivative polynomials of hyperbolic tangent, we obtain a very general convolution identity for the Bernoulli numbers $B_n$ and Genocchi numbers $G_n$, which is used to produce the following one:
\begin{align}
&\sum_{i=0}^{q-1}\binom{q-1}{i}\frac{B_{q+i}G_{2n+q-i}}{(q+i)(2n+q-i)}\nonumber\\
&\quad=-\frac{1}{4}\sum_{i=0}^{2n}(-1)^i\binom{2n}{i}
    \frac{G_{q+i}G_{2n+q-i}}{(q+i)(2n+q-i)}
+\frac{(-1)^q}{q\binom{2q}{q}}\frac{G_{2n+2q}}{2n+2q}
    \,,\quad\text{for }n\geq 0\,,\ q\geq 2\,,\label{id.BG.GG}
\end{align}
so that the coefficients of various constants in the evaluations of the sums in (\ref{intro.sym.TS}), which correspond to the case of $m=q$ and $p$ odd, can be finally determined.

On the other hand, by the definitions, the linear $T$-sums and $\S$-sums are associated with Hoffman's double $t$-values and Kaneko-Tsumura's double $T$-values, respectively:
\begin{align}
T_{p,q}&=\t(q,p)=2^{p+q}t(q,p)
    =2^{p+q-2}\{\ze(q,p)-\ze(q,\bar{p})-\ze(\bar{q},p)+\ze(\bar{q},\bar{p})\}
    \,,\label{Tpq.MtV.MZV}\\
\S_{p,q}&=2^{p+q-2}T(q,p)\nonumber\\
&=(2^{p+q-1}-1)\{\ze(q,p)+\ze(p+q)\}
    +2^{p+q-1}\{\ze(\bar{q},p)+\ze(\ol{p+q})\}\,,\label{Spq.MTV.MZV}
\end{align}
and the linear $R$-sums can be expressed in terms of zeta values, $t$-values and double $T$-values:
\begin{equation}\label{Rpq.MTV.MZV}
R_{p,q}=\ze(p)\t(q)-\S_{q,p}=\ze(p)\t(q)-2^{p+q-2}T(p,q)\,,\quad\text{for }p,q\geq 2
\end{equation}
(see \cite[Eqs. (3.5), (3.7), (3.8) and (3.11)]{WangXu20.DTE}). Therefore, we can transform the results on variants of linear Euler sums to those on variants of double zeta values. For example, a symmetric sum on the double $T$-values can be obtained directly from (\ref{intro.sym.RR}) and (\ref{Rpq.MTV.MZV}):
\[
(-1)^m\la_p(T(m,q))+(-1)^p\la_m(T(p,q))
    \in\mathbb{Q}[\text{zeta values}]\,,\quad\text{for }m,p,q\geq2\,,
\]
which further indicates that the sums $\la_p(T(p,q))$ are reducible to zeta values if $p,q\geq 2$. Additionally, we show that for any even weight $w:=m+q$, with $m,q\geq2$, the double $t$-values $t(q,m)$ and linear $T$-sums $T_{m,q}$ are expressible in terms of MZVs.

The paper is organized as follows. In Section \ref{Sec.Exp.ResTh}, we present some expansions of the parametric digamma function and introduce the residue theorem, which are used in the establishment of the symmetric extensions of the Kaneko-Tsumura conjecture (\ref{KT.conj}). Section \ref{Sec.TS} is devoted to the symmetric sums on the linear $T$-sums and $\S$-sums, and Section \ref{Sec.R} is devoted to the symmetric sums on the linear $R$-sums. Finally, in Section \ref{Sec.con.id}, we prove the convolution identity (\ref{id.BG.GG}) on the Bernoulli numbers and Genocchi numbers, which is required in a proof in Section \ref{Sec.TS}.


\section{Expansions and residue theorem}\label{Sec.Exp.ResTh}

In \cite{Xu19.SEIS}, we introduced a \emph{parametric digamma function} $\vPsi(-s;a)$ by
\[
\vPsi(-s;a)+\gamma
    =\frac{1}{s-a}+\sum_{k=1}^\infty\smbb{\frac{1}{k+a}-\frac{1}{k+a-s}}\,,
    \quad\text{for}\ s\in\mathbb{C}\,,\ a\in\mathbb{C}\setminus\mathbb{Z}^-\,,
\]
where $\ga$ is the \emph{Euler--Mascheroni constant}, and $\mathbb{Z}^-:=\{-1,-2,\ldots\}$. The function $\vPsi(-s;a)$ is meromorphic in the entire complex plane with a simple pole at $s=n+a$ for each nonnegative integer $n$. Here, let
\[
\vPsi(-s):=\vPsi(-s;-\tf{1}{2})+\gamma
    =\frac{1}{s+1/2}+\sum_{k=1}^\infty\smbb{\frac{1}{k-1/2}-\frac{1}{k-1/2-s}}\,.
\]

By \cite[Theorems 2.1--2.3 and Corollary 2.4]{Xu19.SEIS}, the next two lemmas can be established. Firstly, using the special value of the digamma function $\psi(1/2)=-2\ln(2)-\ga$, we have Lemma \ref{Lem.s.n}.

\begin{lemma}\label{Lem.s.n}
For integers $n\geq 0$ and $p\geq 2$, the following expansions hold:
\begin{align}
&\vPsi(\tf{1}{2}-s)
    \stackrel{s\to n}{=}\frac{1}{s-n}+H_n+2\ln(2)
    +\sum_{j=1}^\infty\{(-1)^jH_n^{(j+1)}-\ze(j+1)\}(s-n)^j\,,\label{p1.sn}\\
&\frac{\vPsi^{(p-1)}(\tf{1}{2}-s)}{(p-1)!}
    \stackrel{s\to n}{=}\frac{1}{(s-n)^p}
    +(-1)^p\sum_{j=p}^\infty\binom{j-1}{p-1}\{\ze(j)+(-1)^jH_n^{(j)}\}(s-n)^{j-p}\,.
    \label{p.sn}
\end{align}
\end{lemma}

Next, according to the definitions of the \emph{Hurwitz zeta function} $\ze(s,a+1)=\sum_{k=1}^{\infty}\frac{1}{(k+a)^s}$ and the \emph{parametric harmonic numbers} $H_n^{(s)}(a)=\sum_{k=1}^n\frac{1}{(k+a)^s}$, we have
\[
\ze(s,\tf{1}{2})=\t(s)\,,\quad
H_n^{(s)}(\tf{1}{2})+2^s=h_{n+1}^{(s)}\,,\quad
H_n^{(s)}(-\tf{1}{2})=h_n^{(s)}\,,
\]
which yield Lemma \ref{Lem.s.halfint}.

\begin{lemma}\label{Lem.s.halfint}
For integers $n\geq 1$ and $p\geq 2$, the following expansions hold:
\begin{align}
&\vPsi(\tf{1}{2}-s)
    \stackrel{s\to n-1/2}{=}
    h_n+\sum_{j=1}^\infty\{(-1)^jh_n^{(j+1)}-\t(j+1)\}(s-n+\tf{1}{2})^j
    \,,\label{p1.s.halfint}\\
&\frac{\vPsi^{(p-1)}(\tf{1}{2}-s)}{(p-1)!}
    \stackrel{s\to n-1/2}{=}
    (-1)^p\sum_{j=p}^\infty\binom{j-1}{p-1}
    \{\t(j)+(-1)^jh_n^{(j)}\}(s-n+\tf{1}{2})^{j-p}\,,\label{p.s.halfint}\\
&\vPsi(\tf{1}{2}-s)
    \stackrel{s\to-(n-1/2)}{=}
    h_{n-1}+\sum_{j=1}^\infty\{h_{n-1}^{(j+1)}-\t(j+1)\}(s+n-\tf{1}{2})^j
    \,,\label{p1.s.nhalfint}\\
&\frac{\vPsi^{(p-1)}(\tf{1}{2}-s)}{(p-1)!}
    \stackrel{s\to-(n-1/2)}{=}
    (-1)^p\sum_{j=p}^\infty\binom{j-1}{p-1}
    \{\t(j)-h_{n-1}^{(j)}\}(s+n-\tf{1}{2})^{j-p}\,.\label{p.s.nhalfint}
\end{align}
\end{lemma}

Besides the above two lemmas, by computation, we obtain the next one.

\begin{lemma}\label{Lem.s.nn}
For integers $n\geq 1$ and $p\geq 2$, the following expansions hold:
\begin{align}
&\vPsi(\tf{1}{2}-s)\stackrel{s\to-n}{=}
    H_{n-1}+2\ln(2)+\sum_{j=1}^{\infty}\{H_{n-1}^{(j+1)}-\ze(j+1)\}(s+n)^j
    \,,\label{p1.s.nn}\\
&\frac{\vPsi^{(p-1)}(\tf{1}{2}-s)}{(p-1)!}\stackrel{s\to-n}{=}
    (-1)^{p-1}\sum_{j=p}^{\infty}\binom{j-1}{p-1}\{H_{n-1}^{(j)}-\ze(j)\}(s+n)^{j-p}
    \,.\label{p.s.nn}
\end{align}
\end{lemma}

\begin{proof}
It can be found that
\[
\vPsi(\tf{1}{2}-s)\stackrel{s\to-n}{=}
    2+\sum_{k=0}^{\infty}\smbb{\frac{1}{k-1/2}-\frac{1}{k+n}
    -\frac{1}{k+n}\sum_{j=1}^{\infty}\smbb{\frac{s+n}{k+n}}^j}\,,
\]
which, together with the infinite series
\begin{align*}
\sum_{k=0}^{\infty}\smbb{\frac{1}{k-1/2}-\frac{1}{k+n}}
    &=\sum_{k=0}^{\infty}\smbb{\frac{1}{k-1/2}-\frac{1}{k+1}}
        +\sum_{k=0}^{\infty}\smbb{\frac{1}{k+1}-\frac{1}{k+n}}\\
    &=-2+2\ln(2)+\ga+\psi(n)=-2+2\ln(2)+H_{n-1}\,,
\end{align*}
arising from the properties of the digamma function, gives (\ref{p1.s.nn}). Differentiating (\ref{p1.s.nn}) $p-1$ times with respect to $s$ further leads us to the second expansion of this lemma.
\end{proof}

In particular, if we interpret $\ze(1):=-2\ln(2)$ and $\t(1):=0$ wherever they occur, the expansions (\ref{p.sn}), (\ref{p.s.halfint}), (\ref{p.s.nhalfint}) and (\ref{p.s.nn}) hold for $p=1$. Moreover, it can be found that the expansion (\ref{p.s.halfint}) also hold for $n=0$, for it coincides with the $n=1$ case of (\ref{p.s.nhalfint}).

Finally, due to Flajolet and Salvy's \cite[Lemma 2.1]{FlSa98}, the following residue theorem holds.

\begin{lemma}\label{Lem.Res}
Let $\xi(s)$ be a kernel function and let $r(s)$ be a rational function which is $O(s^{-2})$ at infinity. Then
\[
\sum_{\al\in O}{\rm Res}(r(s)\xi(s),\al)
    +\sum_{\be\in S}{\rm Res}(r(s)\xi(s),\be)=0\,,
\]
where $S$ is the set of poles of $r(s)$ and $O$ is the set of poles of $\xi(s)$ that are not poles of $r(s)$. Here ${\rm Res}(h(s),\la)$ denotes the residue of $h(s)$ at $s=\la$, and the kernel function $\xi(s)$ is meromorphic in the whole complex plane and satisfies $\xi(s)=o(s)$ over an infinite collection of circles $|s|=\rho_k$ with $\rho_k\to+\infty$.
\end{lemma}


\section{Symmetric extension on linear $T$-sums and $\S$-sums}\label{Sec.TS}


\subsection{Main theorem on linear $T$-sums and $\S$-sums}

Let us consider the symmetric extension of the Kaneko-Tsumura conjecture on the linear $T$-sums and $\S$-sums.

\begin{theorem}\label{Th.sym.TS}
For integers $m,p\geq1$ and $q\geq 2$, we have
\begin{align}
&(-1)^{m-1}\sum_{\substack{i+j=p-1\\i,j\geq 0}}
    \binom{m+i-1}{i}\binom{q+j-1}{j}T_{m+i,q+j}\nonumber\\
&\quad+(-1)^{p-1}\sum_{\substack{i+j=m-1\\i,j\geq 0}}
    \binom{p+i-1}{i}\binom{q+j-1}{j}\S_{p+i,q+j}
    \in\mathbb{Q}[\ln(2),\mathrm{zeta\ values}]\,.\label{sym.TS}
\end{align}
In particular, the following expression holds:
\begin{align}
&(-1)^{m-1}\la_p(T_{m,q})+(-1)^{p-1}\la_m(\S_{p,q})\nonumber\\
&\quad=(-1)^m\la_p((-1)^m\t(m)\t(q))
    +(-1)^p\la_m((-1)^p\t(p)\ze(q))
    -\la_q(\ze(m)\t(p))\,,\label{sym.TS.expfor}
\end{align}
where, by our conventions, $\ze(1):=-2\ln(2)$ and $\t(1):=0$ wherever they occur.
\end{theorem}

\begin{proof}
To prove this identity, we consider
\[
\mathcal{G}_1(s):=\frac{\vPsi^{(m-1)}(\tf{1}{2}-s)\vPsi^{(p-1)}(-s)}{(s+1)^q(m-1)!(p-1)!}\,.
\]
The function $\mathcal{G}_1(s)$ has a pole of order $q$ at $s=-1$. By (\ref{p.s.nhalfint}) and (\ref{p.s.nn}), the residue is
\[
{\rm Res}(\mathcal{G}_1(s),-1)=(-1)^{m+p}\la_q(\ze(m)\t(p))\,.
\]
Similarly, $\mathcal{G}_1(s)$ has poles of order $m$ at $s=n$ and poles of order $p$ at $s=n-1/2$ for $n\geq 0$. Then by appealing to the expansions (\ref{p.sn}) and (\ref{p.s.halfint}), the residues are found to be
\[
{\rm Res}(\mathcal{G}_1(s),n)=\sum_{\substack{i+j=m-1\\i,j\geq 0}}
    (-1)^{p+j}\binom{p+i-1}{i}\binom{q+j-1}{j}
    \frac{\t(p+i)+(-1)^{p+i}h_{n+1}^{(p+i)}}{(n+1)^{q+j}}
\]
and
\[
{\rm Res}(\mathcal{G}_1(s),n-\tf{1}{2})=\sum_{\substack{i+j=p-1\\i,j\geq 0}}
    (-1)^{m+j}\binom{m+i-1}{i}\binom{q+j-1}{j}
    \frac{\t(m+i)+(-1)^{m+i}h_n^{(m+i)}}{(n+1/2)^{q+j}}\,,
\]
respectively. Hence, combining these three residue results, applying Lemma \ref{Lem.Res}, and using the definitions of $\ze(s)$, $\t(s)$, $T_{p,q}$ and $\S_{p,q}$, we obtain (\ref{sym.TS.expfor}), which further gives the statement (\ref{sym.TS}) because $\t(s)=(2^s-1)\ze(s)$ for integer $s\geq2$.
\end{proof}

In fact, Theorem \ref{Th.sym.TS} gives an infinite series identity of weight $w=m+p+q-1$. Note that, for odd weights, all the linear $T$-sums $T_{p,q}$ and $\S$-sums $\S_{p,q}$ are already reducible to $\ln(2)$ and zeta values \cite[Corollaries 3.3 and 3.8]{WangXu20.DTE}. Therefore, in this case, the statement (\ref{sym.TS}) is somewhat trivial. However, for even weights, the sums $T_{p,q}$ and $\S_{p,q}$ may be only expressible in terms of (alternating) zeta values and double zeta values (see Eqs. (\ref{Tpq.MtV.MZV}) and (\ref{Spq.MTV.MZV})), but Theorem \ref{Th.sym.TS} asserts that the symmetric sums
\[
(-1)^{m-1}\la_p(T_{m,q})+(-1)^{p-1}\la_m(\S_{p,q})
\]
are still reducible to $\ln(2)$ and zeta values. In Sections \ref{Sec.TS.m1.p1} -- \ref{Sec.TSqeq}, we present several special cases and related examples of this theorem.


\subsection{The case of $m=p=1$}\label{Sec.TS.m1.p1}

When $m=p=1$, Theorem \ref{Th.sym.TS} reduces to the next result.

\begin{corollary}\label{Coro.TS1q}
For integer $q\geq2$, the sums $T_{1,q}+\S_{1,q}$ are reducible to $\ln(2)$ and zeta values:
\begin{equation}\label{TS1q}
T_{1,q}+\S_{1,q}=\sum_{n=1}^{\infty}\frac{h_{n-1}}{(n-1/2)^q}
    +\sum_{n=1}^{\infty}\frac{h_n}{n^q}
    =2\ln(2)\t(q)-\sum_{j=1}^{q-2}\ze(q-j)\t(j+1)\,.
\end{equation}
\end{corollary}

\begin{example}
In Theorem \ref{Th.sym.TS}, replacing $(m,p,q)$ by $(1,1,2)$ -- $(1,1,5)$ yields
\begin{align*}
&T_{1,2}+\S_{1,2}=\pi^2\ln(2)\,,\\
&T_{1,3}+\S_{1,3}=14\ln(2)\ze(3)-\tf{1}{12}\pi^4\,,\\
&T_{1,4}+\S_{1,4}=-\tf{5}{3}\pi^2\ze(3)+\tf{1}{3}\pi^4\ln(2)\,,\\
&T_{1,5}+\S_{1,5}=62\ln(2)\ze(5)-7\ze(3)^2-\tf{1}{30}\pi^6\,,
\end{align*}
respectively, which correspond to the cases of $q=2,3,4,5$ of Corollary \ref{Coro.TS1q}.
Note that the two ones corresponding to $q=3,5$ are of even weights, and the evaluations of the four involved linear $T$-sums and $\S$-sums contain the polylogarithm $\Lii_4(\tf{1}{2})$ and alternating double zeta values $\ze(\bar{5},1)$:
\begin{align*}
&T_{1,3}=-16\Lii_4(\tf{1}{2})-\tf{2}{3}\ln(2)^4+\tf{2}{3}\pi^2\ln(2)^2+\tf{23}{360}\pi^4\,,\\
&T_{1,5}=-32\ze(\bar{5},1)+62\ln(2)\ze(5)+\tf{17}{2}\ze(3)^2-\tf{73}{1260}\pi^6\,,\\
&\S_{1,3}=16\Lii_4(\tf{1}{2})+14\ln(2)\ze(3)
    +\tf{2}{3}\ln(2)^4-\tf{2}{3}\pi^2\ln(2)^2-\tf{53}{360}\pi^4\,,\\
&\S_{1,5}=32\ze(\bar{5},1)-\tf{31}{2}\ze(3)^2+\tf{31}{1260}\pi^6\,.
\end{align*}
These four linear sums have been computed in \cite[Examples 3.5 and 3.8]{WangXu20.DTE}.\hfill\qedsymbol
\end{example}


\subsection{The case of $m=q$ and $p$ odd}

When $m=q\geq 2$ and $p$ is odd, Theorem \ref{Th.sym.TS} reduces to
\begin{align*}
&(-1)^{q-1}\la_p(T_{q,q})+\la_q(\S_{p,q})\\
&\quad=(-1)^q\la_p((-1)^q\t(q)\t(q))+(-1)^p\la_q((-1)^p\t(p)\ze(q))-\la_q(\ze(q)\t(p))\,,
\end{align*}
where the last two terms on the right can be combined into one, as follows:
\begin{align*}
&(-1)^p\la_q((-1)^p\t(p)\ze(q))-\la_q(\ze(q)\t(p))\\
&\quad=\sum_{\substack{i+j=q-1\\i,j\geq 0}}
    \binom{p+i-1}{i}\binom{q+j-1}{j}((-1)^i-1)\t(p+i)\ze(q+j)\\
&\quad=-2\sum_{\substack{i=1\\i\ \text{even}}}^q
    \binom{p+i-2}{p-1}\binom{2q-i-1}{q-i}\t(p+i-1)\ze(2q-i)\,.
\end{align*}
Thus, we obtain
\begin{align}
&(-1)^{q-1}\la_p(T_{q,q})+\la_q(\S_{p,q})\nonumber\\
&\quad=(-1)^q\la_p((-1)^q\t(q)\t(q))
    -2\sum_{k=1}^{[\frac{q}{2}]}\binom{p-2+2k}{p-1}\binom{2q-1-2k}{q-2k}
    \t(p-1+2k)\ze(2q-2k)\,,\label{TS.qeq}
\end{align}
for integer $q\geq2$ and odd integer $p\geq 1$. The further special cases of $p=1,3,5$  of the above identity are of particular interest.

\begin{corollary}\label{Coro.TSq1q}
For integer $q\geq2$, we have
\begin{align*}
(-1)^{q-1}T_{q,q}+\la_q(\S_{1,q})
&=(-1)^{q-1}T_{q,q}+\sum_{j=0}^{q-1}\binom{q+j-1}{j}\S_{q-j,q+j}\\
&=\t(q)^2-2\sum_{k=1}^{[\frac{q}{2}]}
    \binom{2q-1-2k}{q-2k}\t(2k)\ze(2q-2k)\,.
\end{align*}
Therefore, the sums $(-1)^{q-1}T_{q,q}+\la_q(\S_{1,q})$ reduce to rational combinations of $\ze(q)^2$ and $\pi^{2q}$ if $q$ is odd, and to rational multiples of $\pi^{2q}$ if $q$ is even.
\end{corollary}

\begin{proof}
The final assertion arises from the fact that
\[
\t(2k)\ze(2q-2k)=(2^{2k}-1)\ze(2k)\ze(2q-2k)
\]
are rational multiples of $\pi^{2q}$.
\end{proof}

\begin{example}
In Theorem \ref{Th.sym.TS}, replacing $(m,p,q)$ by $(2,1,2)$ -- $(6,1,6)$ yields
\begin{align*}
&T_{2,2}-2\S_{1,3}-\S_{2,2}=-\tf{1}{12}\pi^4\,,\\
&T_{3,3}+6\S_{1,5}+3\S_{2,4}+\S_{3,3}=49\ze(3)^2-\tf{1}{30}\pi^6\,,\\
&T_{4,4}-20\S_{1,7}-10\S_{2,6}-4\S_{3,5}-\S_{4,4}=-\tf{17}{1260}\pi^8\,,\\
&T_{5,5}+70\S_{1,9}+35\S_{2,8}+15\S_{3,7}+5\S_{4,6}+\S_{5,5}
    =961\ze(5)^2-\tf{31}{5670}\pi^{10}\,,\\
&T_{6,6}-252\S_{1,11}-126\S_{2,10}-56\S_{3,9}-21\S_{4,8}-6\S_{5,7}-\S_{6,6}
    =-\tf{691}{311850}\pi^{12}\,,
\end{align*}
respectively. These correspond to the case of $q=2,3,4,5,6$ of Corollary \ref{Coro.TSq1q}.\hfill\qedsymbol
\end{example}

Similarly, when $m=q\geq 2$ and $p=3,5$, Theorem \ref{Th.sym.TS} reduces to the next two results.

\begin{corollary}\label{Coro.TSq3q}
For integer $q\geq2$, we have
\begin{align*}
(-1)^{q-1}\la_3(T_{q,q})+\la_q(\S_{3,q})
&=q(q+1)\t(q)\t(q+2)-q^2\t(q+1)^2\\
&\quad-2\sum_{k=1}^{[\frac{q}{2}]}
    \binom{2k+1}{2}\binom{2q-1-2k}{q-2k}\t(2k+2)\ze(2q-2k)\,.
\end{align*}
Therefore, the sums $(-1)^{q-1}\la_3(T_{q,q})+\la_q(\S_{3,q})$ are reducible to rational combinations of $\ze(q)\ze(q+2)$ and $\pi^{2q+2}$ if $q$ is odd, and to rational combinations of $\ze(q+1)^2$ and $\pi^{2q+2}$ if $q$ is even.
\end{corollary}

\begin{corollary}\label{Coro.TSq5q}
For integer $q\geq2$, we have
\begin{align*}
&(-1)^{q-1}\la_5(T_{q,q})+\la_q(\S_{5,q})\\
&\quad=2\binom{q+3}{4}\t(q)\t(q+4)-2q\binom{q+2}{3}\t(q+1)\t(q+3)
    +\binom{q+1}{2}^2\t(q+2)^2\\
&\quad\quad-2\sum_{k=1}^{[\frac{q}{2}]}
    \binom{2k+3}{4}\binom{2q-1-2k}{q-2k}\t(2k+4)\ze(2q-2k)\,.
\end{align*}
Therefore, the sums $(-1)^{q-1}\la_5(T_{q,q})+\la_q(\S_{5,q})$ are reducible to rational combinations of $\ze(q)\ze(q+4)$, $\ze(q+2)^2$ and $\pi^{2q+4}$ if $q$ is odd, and to rational combinations of $\ze(q+1)\ze(q+3)$ and $\pi^{2q+4}$ if $q$ is even.
\end{corollary}

\begin{example}
In Theorem \ref{Th.sym.TS}, replacing $(m,p,q)$ by $(2,3,2)$, $(3,3,3)$, $(4,3,4)$ gives
\begin{align*}
&3T_{2,4}+4T_{3,3}+3T_{4,2}-2\S_{3,3}-3\S_{4,2}=196\ze(3)^2-\tf{1}{3}\pi^6\,,\\
&2T_{3,5}+3T_{4,4}+2T_{5,3}+2\S_{3,5}+3\S_{4,4}+2\S_{5,3}
    =868\ze(3)\ze(5)-\tf{17}{180}\pi^8\,,\\
&5T_{4,6}+8T_{5,5}+5T_{6,4}-10\S_{3,7}-15\S_{4,6}-12\S_{5,5}-5\S_{6,4}
    =7688\ze(5)^2-\tf{31}{315}\pi^{10}\,,
\end{align*}
respectively. Replacing $(m,p,q)$ by $(2,5,2)$, $(3,5,3)$ yields
\begin{align*}
&5T_{2,6}+8T_{3,5}+9T_{4,4}+8T_{5,3}+5T_{6,2}-2\S_{5,3}-5\S_{6,2}
    =3472\ze(3)\ze(5)-\tf{17}{36}\pi^8\,,\\
&5T_{3,7}+10T_{4,6}+12T_{5,5}+10T_{6,4}+5T_{7,3}+2\S_{5,5}+5\S_{6,4}+5\S_{7,3}\\
&\quad=8890\ze(3)\ze(7)+11532\ze(5)^2-\tf{31}{135}\pi^{10}\,,
\end{align*}
respectively. These can also be obtained directly from Corollaries \ref{Coro.TSq3q} and \ref{Coro.TSq5q}.\hfill\qedsymbol
\end{example}


\subsection{More discussions on the case of $m=q$ and $p$ odd}\label{Sec.m.TSqod}

It will be interesting to give an explicit characterization for the coefficients of various mathematical constants involved in the reduction of the sums $(-1)^{q-1}\la_p(T_{q,q})+\la_q(\S_{p,q})$ appeared in (\ref{TS.qeq}) in the last section.

To do this, let $B_{n}$ be the well-known \emph{Bernoulli numbers} and $G_n$ be the \emph{Genocchi numbers}, defined by
\[
\frac{t}{\ue^t-1}=\sum_{n=0}^{\infty}B_n\frac{t^n}{n!}
\quad\text{and}\quad
\frac{2t}{\ue^t+1}=\sum_{n=1}^{\infty}G_n\frac{t^n}{n!}\,,
\]
respectively (see, for example, \cite[Section 1.14]{Com74}). Then $G_n=2(1-2^n)B_n$ for $n\geq 0$, and $B_{2k+1}=G_{2k+1}=0$ for $k\geq 1$. As mentioned in Section \ref{Sec.intro}, the convolution identity (\ref{id.BG.GG}) on the Bernoulli numbers and Genocchi numbers can be established, which helps us obtain the next theorem. Note that the proof of identity (\ref{id.BG.GG}) will be given in the last section (i.e., Section \ref{Sec.con.id}) of the present paper.

\begin{theorem}\label{Th.TS.qeq.coeff}
For integer $q\geq2$ and odd integer $p\geq 1$, the following explicit expression holds:
\begin{align}
&(-1)^{q-1}\la_p(T_{q,q})+\la_q(\S_{p,q})\nonumber\\
&\quad=(-1)^{q-1}\sum_{\substack{i=0\\q+i\ \mathrm{odd}}}^{p-1}
    \binom{q+i-1}{i}\binom{p+q-2-i}{p-1-i}\t(q+i)\t(p+q-1-i)\nonumber\\
&\quad\quad+\frac{(-1)^{\frac{p-1}{2}}}{8}\binom{p-2+2q}{p-1}
        \frac{G_{p-1+2q}}{(p-1+2q)!}(2\pi)^{p-1+2q}\,.\label{TS.qeq.coeff}
\end{align}
\end{theorem}

\begin{proof}
By appealing to the values of the Riemann zeta function at even positive integers:
\[
\ze(2k)=\frac{(-1)^{k+1}B_{2k}(2\pi)^{2k}}{2\cdot(2k)!}\,,
\]
the second term on the right of Eq. (\ref{TS.qeq}), abbreviated as $\vOm_2$, can be rewritten as
\begin{align}
\vOm_2
&=-2\sum_{k=1}^{[\frac{q}{2}]}\binom{p-2+2k}{p-1}\binom{2q-1-2k}{q-2k}
    (2^{p-1+2k}-1)\ze(p-1+2k)\ze(2q-2k)\nonumber\\
&=\frac{(-1)^{\frac{p-1+2q}{2}}}{2}(2\pi)^{p-1+2q}
    \sum_{k=1}^{[\frac{q}{2}]}\binom{p-2+2k}{p-1}\binom{2q-1-2k}{q-2k}
    \frac{(1-2^{p-1+2k})B_{p-1+2k}B_{2q-2k}}{(p-1+2k)!(2q-2k)!}\nonumber\\
&=\frac{(-1)^{\frac{p-1+2q}{2}}}{4}
    \frac{(2\pi)^{p-1+2q}}{(p-1)!(2q)!}\binom{2q}{q}q^2
    \sum_{k=1}^{[\frac{q}{2}]}\binom{q-1}{2k-1}\frac{G_{p-1+2k}B_{2q-2k}}{(p-1+2k)(2q-2k)}
    \,.\label{2nd.term}
\end{align}
By considering the fact $B_{2k+1}=G_{2k+1}=0$ for $k\geq 1$, the convolution identity (\ref{id.BG.GG}) can be rewritten as
\begin{align*}
&\sum_{k=1}^{[\frac{q}{2}]}\binom{q-1}{2k-1}\frac{G_{p-1+2k}B_{2q-2k}}{(p-1+2k)(2q-2k)}
    \nonumber\\
&\quad=\frac{(-1)^q}{q\binom{2q}{q}}\frac{G_{p-1+2q}}{p-1+2q}
    +\frac{1}{4}\sum_{\substack{i=0\\q+i\ \mathrm{even}}}^{p-1}(-1)^{i+1}\binom{p-1}{i}
        \frac{G_{q+i}G_{p+q-1-i}}{(q+i)(p+q-1-i)}\,,
\end{align*}
where $p\geq 1$ is odd, and $q\geq2$. Substituting it into the right of (\ref{2nd.term}), and using the relation
\[
G_{2k}=2(1-2^{2k})B_{2k}=\frac{(-1)^k\cdot4\t(2k)\cdot(2k)!}{(2\pi)^{2k}}\,,
\]
we have
\begin{align*}
\vOm_2
&=\frac{(-1)^{\frac{p-1}{2}}}{8}\binom{p-2+2q}{p-1}
    \frac{G_{p-1+2q}}{(p-1+2q)!}(2\pi)^{p-1+2q}\\
&\quad\quad-\sum_{\substack{i=0\\q+i\ \mathrm{even}}}^{p-1}
    (-1)^i\binom{q+i-1}{i}\binom{p+q-2-i}{p-1-i}
        \t(q+i)\t(p+q-1-i)\,.
\end{align*}
On the other hand, the first term on the right of Eq. (\ref{TS.qeq}) equals
\begin{align*}
\vOm_1&=(-1)^q\la_p((-1)^q\t(q)\t(q))\\
&=(-1)^q\sum_{\substack{i+j=p-1\\i,j\geq 0}}
    \binom{q+i-1}{i}\binom{q+j-1}{j}(-1)^{q+i}\t(q+i)\t(q+j)\\
&=\bibb{\sum_{\substack{i=0\\q+i\ \mathrm{odd}}}^{p-1}+
    \sum_{\substack{i=0\\q+i\ \mathrm{even}}}^{p-1}}
        (-1)^i\binom{q+i-1}{i}\binom{p+q-2-i}{p-1-i}\t(q+i)\t(p+q-1-i)\,.
\end{align*}
Thus, all the summands with even $q+i$ will be eliminated in $\vOm_1+\vOm_2$, and we obtain (\ref{TS.qeq.coeff}).
\end{proof}

By Theorem \ref{Th.TS.qeq.coeff}, the coefficients of various mathematical constants in the evaluations of the sums $(-1)^{q-1}\la_p(T_{q,q})+\la_q(\S_{p,q})$ can be completely determined. For example, setting $p=1$ in Eq. (\ref{TS.qeq.coeff}) gives
\[
(-1)^{q-1}T_{q,q}+\sum_{j=0}^{q-1}\binom{q+j-1}{j}\S_{q-j,q+j}
=\frac{1-(-1)^q}{2}(2^q-1)^2\ze(q)^2+\frac{G_{2q}}{8\cdot(2q)!}(2\pi)^{2q}\,,
\]
and setting $p=3$ yields
\begin{align*}
&(-1)^{q-1}\la_3(T_{q,q})+\la_q(\S_{3,q})\nonumber\\
&\quad=\frac{1-(-1)^q}{2}q(q+1)(2^q-1)(2^{q+2}-1)\ze(q)\ze(q+2)
    -\frac{1+(-1)^q}{2}q^2(2^{q+1}-1)^2\ze(q+1)^2\\
&\quad\quad-q(2q+1)\frac{G_{2q+2}}{8\cdot(2q+2)!}(2\pi)^{2q+2}\,.
\end{align*}


\subsection{The case of $m=q$ and $p$ even}\label{Sec.TSqeq}

Here, we present some interesting cases of Theorem \ref{Th.sym.TS} with odd weights. For example, setting $(m,p,q)$ by $(2, 1, 3)$ and $(3, 1, 2)$ yields the following two ones:
\begin{align*}
&T_{2,3}-3\S_{1,4}-\S_{2,3}=-\tf{5}{6}\pi^2\ze(3)\,,\\
&T_{3,2}+3\S_{1,4}+2\S_{2,3}+\S_{3,2}=\tf{19}{6}\pi^2\ze(3)\,,
\end{align*}
which are reducible to rational multiples of $\pi^2\ze(3)$. Moreover, the special cases of Theorem \ref{Th.sym.TS} with the mode of $(m,p,q)=(q,2n,q)$ also deserve attention.

\begin{corollary}\label{Coro.TSqeq}
For integers $q\geq 2$ and $n\geq 1$, the sums $(-1)^{q-1}\la_{2n}(T_{q,q})-\la_q(\S_{2n,q})$ are expressible in terms of $\ze(2n+i)\pi^{2q-1-i}$, where $i$ is odd and satisfies $1\leq i\leq q-1$. In particular, we have
\begin{align}
&(-1)^{q-1}\la_{2n}(T_{q,q})-\la_q(\S_{2n,q})\nonumber\\
&\quad=-2\sum_{\substack{i+j=q-1\\i\emph{ odd}}}
    \binom{2n+i-1}{i}\binom{q+j-1}{j}\t(2n+i)\ze(q+j)\,.\label{sym.TS.qeq}
\end{align}
\end{corollary}

\begin{proof}
When $m=q$ and $p=2n$, by Theorem \ref{Th.sym.TS}, the left side of Eq. (\ref{sym.TS.qeq}) equals
\[
\text{LHS}=(-1)^q\la_{2n}((-1)^q\t(q)\t(q))
    +\la_q((-1)^{2n}\t(2n)\ze(q))-\la_q(\ze(q)\t(2n))\,.
\]
Note that the sum
\begin{align*}
&(-1)^q\la_{2n}((-1)^q\t(q)\t(q))\\
&\quad=\bibb{\sum_{i=0}^{n-1}+\sum_{i=n}^{2n-1}}
    \binom{q+i-1}{q-1}\binom{q+2n-2-i}{q-1}(-1)^i\t(q+i)\t(q+2n-1-i)
\end{align*}
will vanish by changing of the variable $i\to 2n-1-i$ in the second term, and
\begin{align*}
&\la_q((-1)^{2n}\t(2n)\ze(q))-\la_q(\ze(q)\t(2n))\\
&\quad=\sum_{\substack{i+j=q-1\\i,j\geq 0}}\binom{2n+i-1}{i}\binom{q+j-1}{j}
    ((-1)^i-1)\t(2n+i)\ze(q+j)\,.
\end{align*}
Thus, we obtain Eq. (\ref{sym.TS.qeq}). Finally, when $i$ is odd, $q+j=2q-1-i$ is even, so the assertion in the corollary also holds.
\end{proof}

It is obvious that there are $[q/2]$ terms in the right side of (\ref{sym.TS.qeq}). Here, we present the further cases of $q=2,3,4,5$, which satisfy $[q/2]<3$.

\begin{corollary}\label{Coro.TS2e2}
For integer $n\geq1$, the sums $\la_{2n}(T_{2,2})+\la_2(\S_{2n,2})$ reduce to rational multiples of $\pi^2\ze(2n+1)$. In particular, we have
\[
\frac{1}{2}\sum_{\substack{i+j=2n-1\\i,j\geq 0}}(i+1)(j+1)T_{i+2,j+2}
        +\{n\S_{2n+1,2}+\S_{2n,3}\}=\frac{n(2^{2n+1}-1)}{3}\pi^2\ze(2n+1)\,.
\]
\end{corollary}

\begin{proof}
It follows from the identity $\la_{2n}(T_{2,2})+\la_2(\S_{2n,2})=4n\ze(2)\t(2n+1)$, which corresponds to the case $(m,p,q)=(2,2n,2)$.
\end{proof}

\begin{example}
Replacing $(m,p,q)$ in Theorem \ref{Th.sym.TS} by $(2,2,2)$, $(2,4,2)$, $(2,6,2)$ yields
\begin{align*}
&T_{2,3}+T_{3,2}+\S_{2,3}+\S_{3,2}=\tf{7}{3}\pi^2\ze(3)\,,\\
&2T_{2,5}+3T_{3,4}+3T_{4,3}+2T_{5,2}+\S_{4,3}+2\S_{5,2}=\tf{62}{3}\pi^2\ze(5)\,,\\
&3T_{2,7}+5T_{3,6}+6T_{4,5}+6T_{5,4}+5T_{6,3}+3T_{7,2}+\S_{6,3}+3\S_{7,2}=127\pi^2\ze(7)\,.
\end{align*}
These can also be obtained from Corollary \ref{Coro.TS2e2} by setting $n=1,2,3$. \hfill\qedsymbol
\end{example}

The following three identities give special cases of $(m,p,q)=(3,2n,3),(4,2n,4),(5,2n,5)$, respectively, and can be derived directly from Corollary \ref{Coro.TSqeq}.

\begin{corollary}
For integer $n\geq1$, we have
\begin{align*}
&\la_{2n}(T_{3,3})-\la_3(\S_{2n,3})=-12n\ze(4)\t(2n+1)\,,\\
&\la_{2n}(T_{4,4})+\la_{4}(\S_{2n,4})
    =40n\ze(6)\t(2n+1)+\tf{1}{3}(2n+2)(2n+1)(2n)\ze(4)\t(2n+3)\,,\\
&\la_{2n}(T_{5,5})-\la_{5}(\S_{2n,5})
    =-140n\ze(8)\t(2n+1)-\tf{5}{3}(2n+2)(2n+1)(2n)\ze(6)\t(2n+3)\,.
\end{align*}
Therefore, the sums $(-1)^{q-1}\la_{2n}(T_{q,q})-\la_q(\S_{2n,q})$ reduce to rational multiples of $\pi^4\ze(2n+1)$ if $q=3$. Moreover, they are reducible to combinations of $\pi^6\ze(2n+1)$ and $\pi^4\ze(2n+3)$ if $q=4$, and to combinations of $\pi^8\ze(2n+1)$ and $\pi^6\ze(2n+3)$ if $q=5$.
\end{corollary}

\begin{example}
Replacing $(m,p,q)$ by $(3,2,3)$, $(3,4,3)$ in Theorem \ref{Th.sym.TS} yields
\begin{align*}
&T_{3,4}+T_{4,3}-2\S_{2,5}-2\S_{3,4}-\S_{4,3}= -\tf{14}{45}\pi^4\ze(3)\,,\\
&5T_{3,6}+9T_{4,5}+9T_{5,4}+5T_{6,3}-3\S_{4,5}-6\S_{5,4}-5\S_{6,3}=-\tf{62}{15}\pi^4\ze(5)\,,
\end{align*}
respectively. Replacing $(m,p,q)$ by $(4,2,4)$, $(5,2,5)$ yields
\begin{align*}
&T_{4,5}+T_{5,4}+5\S_{2,7}+5\S_{3,6}+3\S_{4,5}+\S_{5,4}
    =\tf{31}{45}\pi^4\ze(5)+\tf{2}{27}\pi^6\ze(3)\,,\\
&T_{5,6}+T_{6,5}-14\S_{2,9}-14\S_{3,8}-9\S_{4,7}-4\S_{5,6}-\S_{6,5}
    =-\tf{248}{945}\pi^6\ze(5)-\tf{14}{675}\pi^8\ze(3)\,,
\end{align*}
respectively.\hfill\qedsymbol
\end{example}


\subsection{Symmetric sum on double $T$-values and double $t$-values}

By the relations
\[
\S_{p,q}=2^{p+q-2}T(q,p)\quad\text{and}\quad
    T_{p,q}=\t(q,p)=2^{p+q}t(q,p)\,,
\]
we can transform Theorem \ref{Th.sym.TS} into the following one on Hoffman's double $t$-values and Kaneko-Tsumura's double $T$-values.

\begin{theorem}\label{Th.sym.TtV}
For integers $m,p\geq1$ and $q\geq 2$, we have
\[
(-1)^{m-1}2^{m+p+q-1}\la_p(t(q,m))+(-1)^{p-1}2^{m+p+q-3}\la_m(T(q,p))
    \in\mathbb{Q}[\ln(2),\mathrm{zeta\ values}]\,.
\]
In particular, the following expression holds:
\begin{align}
&(-1)^{m-1}2^{m+p+q-1}\la_p(t(q,m))
    +(-1)^{p-1}2^{m+p+q-3}\la_m(T(q,p))\nonumber\\
&\quad=(-1)^m\la_p((-1)^m\t(m)\t(q))
    +(-1)^p\la_m((-1)^p\t(p)\ze(q))
    -\la_q(\ze(m)\t(p))\,,\label{sym.Ttv.expfor}
\end{align}
where, by our conventions, $\ze(1):=-2\ln(2)$ and $\t(1):=0$ wherever they occur.
\end{theorem}

\begin{proof}
It follows from the transformation formulas
\begin{equation}\label{la.MtTV.STsum}
\la_p(T_{m,q})=2^{m+p+q-1}\la_p(t(q,m))\,,\quad
\la_m(\S_{p,q})=2^{m+p+q-3}\la_m(T(q,p))\,,
\end{equation}
and Eq. (\ref{sym.TS.expfor}).
\end{proof}

In the next corollary, we show that by Theorem \ref{Th.sym.TtV}, similar results to (\ref{KT.conj}) also hold for double $t$-values, linear $T$-sums and linear $\S$-sums, from which we can further show a reduction property of the double $t$-values and linear $T$-sums.

\begin{corollary}\label{Coro.Ts.dtv.MZV}
For any even weight $w:=m+q$, where $m,q\geq2$, the double $t$-values $t(q,m)$ and linear $T$-sums $T_{m,q}$ are expressible in terms of MZVs.
\end{corollary}

\begin{proof}
According to the Kaneko-Tsumura conjecture (\ref{KT.conj}) (see also \cite[Theorem 42]{Mura21}), for integers $m,q\geq2$ and $p\geq 1$, with $m+p+q-1$ even, the sums $\la_m(T(q,p))\in\mathcal{Z}$, so $\la_m(\S_{p,q})\in\mathcal{Z}$ by (\ref{la.MtTV.STsum}). On the other hand, if $m,q\geq 2$, we have
\[
(-1)^{m-1}2^{m+p+q-1}\la_p(t(q,m))+(-1)^{p-1}2^{m+p+q-3}\la_m(T(q,p))
    \in\mathbb{Q}[\mathrm{zeta\ values}]\,,
\]
and $\ln(2)$ does not appear in the evaluations of the sums on the left. Thus, when $m,q\geq 2$, $p\geq 1$, and $m+p+q-1$ is even, we have
\begin{equation}
\la_p(t(q,m)),\ \la_p(T_{m,q})\in\mathcal{Z}\,.
\end{equation}
The final assertion of this corollary follows by setting $p=1$.
\end{proof}

\begin{remark}
Corollary \ref{Coro.Ts.dtv.MZV} can be compared with \cite[Corollary 3.3]{WangXu20.DTE}, where the latter shows that $T_{p,q}$ and $t(q,p)$ are reducible to $\ln(2)$ and zeta values if $p\geq 1$, $q\geq 2$, and $p+q$ is odd. Corollary \ref{Coro.Ts.dtv.MZV} is also a special case of Murakami's recent result \cite[Theorem 1]{Mura21}, which shows that when all $s_j\geq 2$, the MtVs $t(s_1,\ldots,s_k)$ are expressible in terms of MZVs.\hfill\qedsymbol
\end{remark}


\section{Symmetric extension on linear $R$-sums}\label{Sec.R}


\subsection{Main theorem on linear $R$-sums}

In this section, we present a symmetric extension of the Kaneko-Tsumura conjecture (\ref{KT.conj}) on the linear $R$-sums:
\[
R_{p,q}:=\sum_{n=1}^{\infty}\frac{H_{n-1}^{(p)}}{(n-1/2)^q}\,,\quad\text{for }q\geq 2\,.
\]

\begin{theorem}\label{Th.sym.R}
For integers $m,p\geq 1$ and $q\geq2$, we have
\begin{align}
&(-1)^{m-1}\sum_{\substack{i+j=p-1\\i,j\geq 0}}
    \binom{m+i-1}{i}\binom{q+j-1}{j}R_{m+i,q+j}\nonumber\\
&\quad+(-1)^{p-1}\sum_{\substack{i+j=m-1\\i,j\geq 0}}
    \binom{p+i-1}{i}\binom{q+j-1}{j}R_{p+i,q+j}
    \in\mathbb{Q}[\ln(2),\mathrm{zeta\ values}]\,.\label{sym.R}
\end{align}
In particular, the following expression holds:
\begin{align}
&(-1)^{m-1}\la_p(R_{m,q})+(-1)^{p-1}\la_m(R_{p,q})\nonumber\\
&\quad=\binom{m+p+q-2}{q-1}\t(m+p+q-1)
    +(-1)^m\la_p((-1)^m\ze(m)\t(q))\nonumber\\
&\quad\quad+(-1)^p\la_m((-1)^p\ze(p)\t(q))
    -\la_q(\t(m)\t(p))\,,\label{sym.R.expfor}
\end{align}
where, by our conventions, $\ze(1):=-2\ln(2)$ and $\t(1):=0$ wherever they occur.
\end{theorem}

\begin{proof}
The proof of this theorem is similar to that of Theorem \ref{Th.sym.TS}. Now, consider the function
\[
\mathcal{G}_2(s):=
    \frac{\vPsi^{(m-1)}(\tf{1}{2}-s)\vPsi^{(p-1)}(\tf{1}{2}-s)}
        {(s+\tf{1}{2})^q(m-1)!(p-1)!}\,.
\]
It is obvious that the only singularities are poles at $s=-1/2$ and $s=n$ for $n\geq0$. By (\ref{p.s.halfint}), the pole at $-1/2$ has order $q$, and the residue is
\[
{\rm Res}(\mathcal{G}_2(s),-\tf{1}{2})=(-1)^{m+p}\la_q(\t(m)\t(p))\,.
\]
Next, by (\ref{p.sn}), the pole at a nonnegative integer $n$ has order $m+p$, and the residue is
\begin{align*}
&{\rm Res}(\mathcal{G}_2(s),n)\\
&\quad=(-1)^{m+p-1}\binom{m+p+q-2}{q-1}\frac{1}{(n+\tf{1}{2})^{m+p+q-1}}\\
&\quad\quad+(-1)^{m+p-1}\sum_{\substack{i+j=p-1\\i,j\geq 0}}\binom{m+i-1}{i}\binom{q+j-1}{j}
    \frac{(-1)^{i}\ze(m+i)+(-1)^{m}H_n^{(m+i)}}{(n+\tf{1}{2})^{q+j}}\\
&\quad\quad+(-1)^{m+p-1}\sum_{\substack{i+j=m-1\\i,j\geq 0}}\binom{p+i-1}{i}\binom{q+j-1}{j}
    \frac{(-1)^{i}\ze(p+i)+(-1)^{p}H_n^{(p+i)}}{(n+\tf{1}{2})^{q+j}}\,.
\end{align*}
Hence, summing these two contributions, considering the definitions of $\t(s)$ and $R_{p,q}$, and doing some transformations, we obtain the desired formula (\ref{sym.R.expfor}), and therefore the statement.
\end{proof}

Now, let us briefly discuss some special cases of this theorem. Setting $m=p=1$ in Eq. (\ref{sym.R.expfor}) yields an expression of the sums $R_{1,q}$, and an alternate way to obtain this expression is to set $a=-1/2$ in \cite[Theorem 3.2]{Xu19.SEIS}.

\begin{corollary}
For integer $q\geq2$, the linear sums $R_{1,q}$ are reducible to $\ln(2)$ and zeta values:
\begin{equation}\label{R1q}
R_{1,q}=\sum_{n=1}^{\infty}\frac{H_{n-1}}{(n-1/2)^q}
    =\frac{q}{2}\t(q+1)-2\ln(2)\t(q)-\frac{1}{2}\sum_{j=1}^{q-2}\t(q-j)\t(j+1)\,.
\end{equation}
\end{corollary}

\begin{example}
The first few sums are
\begin{align*}
&R_{1,2}=7\ze(3)-\pi^2\ln(2)\,,\\
&R_{1,3}=-14\ln(2)\ze(3)+\tf{1}{8}\pi^4\,,\\
&R_{1,4}=62\ze(5)-\tf{7}{2}\pi^2\ze(3)-\tf{1}{3}\pi^4\ln(2)\,,\\
&R_{1,5}=-62\ln(2)\ze(5)-\tf{49}{2}\ze(3)^2+\tf{1}{12}\pi^6\,,
\end{align*}
which can also be computed from Theorem \ref{Th.sym.R} by replacing $(m,p,q)$ by $(1,1,2)$ -- $(1,1,5)$. The values of more sums can be obtained by specifying the parameter $q$ directly.\hfill\qedsymbol
\end{example}

More generally, setting $m=p\geq 2$ in Eq. (\ref{sym.R.expfor}), we have

\begin{corollary}\label{Coro.la.R}
For integers $p,q\geq2$, the sums $\la_p(R_{p,q})$ are reducible to zeta values:
\begin{align*}
\la_p(R_{p,q})
&=\sum_{\substack{i+j=p-1\\i,j\geq 0}}\binom{p+i-1}{i}\binom{q+j-1}{j}R_{p+i,q+j}\\
&=\frac{(-1)^{p-1}}{2}
    \bibb{\binom{2p+q-2}{q-1}\t(2p+q-1)-\la_q(\t(p)\t(p))}-\la_p((-1)^p\ze(p)\t(q))\,.
\end{align*}
\end{corollary}

As illustrated in \cite[Eq. (3.9)]{WangXu20.DTE} and \cite[Corollary 3.2]{Xu20.EFS}, besides the linear sums $R_{1,q}$, the sums $R_{p,q}$ with $p+q$ odd are also reducible to $\ln(2)$ and zeta values. Therefore, we show here some more special cases of Theorem \ref{Th.sym.R} and Corollary \ref{Coro.la.R} with even weights.

\begin{example}
Setting $(m,p,q)$ by $(2,2,3)$, $(2,2,5)$ in Theorem \ref{Th.sym.R} or Corollary \ref{Coro.la.R}, we have
\begin{align*}
&3R_{2,4}+2R_{3,3}
    =3\sum_{n=1}^\infty\frac{H_{n-1}^{(2)}}{(n-1/2)^4}
        +2\sum_{n=1}^\infty\frac{H_{n-1}^{(3)}}{(n-1/2)^3}
    =112\ze(3)^2-\tf{1}{6}\pi^6\,,\\
&5R_{2,6}+2R_{3,5}=1798\ze(3)\ze(5)-\tf{17}{72}\pi^8\,.
\end{align*}
However, by (\ref{Spq.MTV.MZV}) and (\ref{Rpq.MTV.MZV}), it can be found that (alternating) double zeta values appear in the evaluations of the involved $R$-sums:
\begin{align*}
&R_{2,4}=128\ze(\bar{5},1)+\ze(3)^2-\tf{1}{210}\pi^6\,,\\
&R_{3,3}=-192\ze(\bar{5},1)+\tf{109}{2}\ze(3)^2-\tf{8}{105}\pi^6\,,\\
&R_{2,6}=768\zeta(\bar{7},1)+289\ze(6,2)-864\ze(3)\ze(5)+\tf{59}{525}\pi^8\,,\\
&R_{3,5}=-1920\ze(\bar{7},1)-\tf{1445}{2}\ze(6,2)+3059\ze(3)\ze(5)-\tf{2011}{5040}\pi^8\,.
\end{align*}
See also the evaluations in \cite[Example 3.13]{WangXu20.DTE}, which are obtained by colored multiple zeta values. Similarly, more relations can be established. For example, let $(m,p,q)$ by $(3,3,5)$ and $(4,4,5)$. Then we have
\begin{align*}
&5R_{3,7}+5R_{4,6}+2R_{5,5}=-3810\ze(3)\ze(7)-5704\ze(5)^2+\tf{31}{270}\pi^{10}\,,\\
&7R_{4,8}+12R_{5,7}+10R_{6,6}+4R_{7,5}=64640\ze(5)\ze(7)-\tf{691}{9450}\pi^{12}\,,
\end{align*}
respectively.\hfill\qedsymbol
\end{example}

\begin{example}
Finally, we present another two special cases:
\begin{align*}
&9R_{2,10}+2R_{3,9}=58254\ze(3)\ze(9)+94488\ze(5)\ze(7)-\tf{691}{3780}\pi^{12}\,,\\
&42R_{4,10}+56R_{5,9}+35R_{6,8}+10R_{7,7}
    =1802808\ze(5)\ze(9)+1614170\ze(7)^2-\tf{5461}{14175}\pi^{14}\,,
\end{align*}
which correspond to the cases of $(m,p,q)=(2,3,8),(4,5,6)$ of Theorem \ref{Th.sym.R}, respectively.\hfill\qedsymbol
\end{example}


\subsection{Symmetric sum on double $T$-values}

By substituting the relations (\ref{Rpq.MTV.MZV}), we can use Theorem \ref{Th.sym.R} to establish a symmetric extension of the Kaneko-Tsumura conjecture (\ref{KT.conj}) on the double $T$-values.

\begin{theorem}\label{Th.KTC.T}
For integers $m,p,q\geq 2$, we have
\[
(-1)^m\la_p(T(m,q))+(-1)^p\la_m(T(p,q))\in\mathbb{Q}[\mathrm{zeta\ values}].
\]
In particular, the following expression holds:
\begin{align}
&2^{m+p+q-3}\{(-1)^m\la_p(T(m,q))+(-1)^p\la_m(T(p,q))\}\nonumber\\
&\quad=\binom{m+p+q-2}{q-1}\t(m+p+q-1)
    +(-1)^m\la_p((1+(-1)^m)\ze(m)\t(q))\nonumber\\
&\quad\quad+(-1)^p\la_m((1+(-1)^p)\ze(p)\t(q))
    -\la_q(\t(m)\t(p))\,.\label{sym.T.expfor}
\end{align}
\end{theorem}

Taking $m=p$ in Theorem \ref{Th.KTC.T} further yields the next corollary.

\begin{corollary}
For integers $p,q\geq 2$, the sums $\la_p(T(p,q))$ are reducible to zeta values:
\begin{align*}
\la_p(T(p,q))&=\frac{(-1)^p}{2^{2p+q-2}}\binom{2p+q-2}{q-1}\t(2p+q-1)
    +\frac{1}{2^{2p+q-3}}\la_p((1+(-1)^p)\ze(p)\t(q))\\
    &\quad-\frac{(-1)^p}{2^{2p+q-2}}\la_q(\t(p)\t(p))\,.
\end{align*}
\end{corollary}

By specifying the parameters, a series of relations on double $T$-values can be established.

\begin{example}
For example, in Theorem \ref{Th.sym.R}, replacing $(m,p,q)$ by $(5,4,2)$, $(5,6,2)$ gives two relations on the linear $R$-sums:
\begin{align*}
&5R_{4,6}+12R_{5,5}+15R_{6,4}+10R_{7,3}
    =70\ze(3)\ze(7)+4216\ze(5)^2-\tf{31}{630}\pi^{10}\,,\\
&3R_{5,7}+10R_{6,6}+18R_{7,5}+21R_{8,4}+14R_{9,3}
    =-10872\ze(5)\ze(7)+98\ze(3)\ze(9)+\tf{691}{56700}\pi^{12}\,;
\end{align*}
while for the double $T$-values, substituting the same parameters into Theorem \ref{Th.KTC.T} yields
\begin{align*}
&5T(4,6)+12T(5,5)+15T(6,4)+10T(7,3)=-\tf{961}{64}\ze(5)^2+\tf{1}{4608}\pi^{10}\,,\\
&3T(5,7)+10T(6,6)+18T(7,5)+21T(8,4)+14T(9,3)
    =\tf{11811}{1024}\ze(5)\ze(7)-\tf{1}{92160}\pi^{12}\,,
\end{align*}
respectively.\hfill\qedsymbol
\end{example}


\section{Proof of the required convolution identity}\label{Sec.con.id}

In this section, we give the proof of the convolution identity (\ref{id.BG.GG}) on the Bernoulli numbers and Genocchi numbers used in the discussions of Section \ref{Sec.m.TSqod}.

Firstly, it can be found that there are polynomials $P_n(y)$ of degree $n+1$ with integer coefficients, so that
\[
P_n(\tanh(t))=\mathcal{D}_t^{n}\tanh(t)\,,\quad\text{for } n=0,1,2,\ldots,
\]
where $\mathcal{D}_t$ is the derivative operator defined by $\mathcal{D}_tf(t)=f'(t)$. In particular,
\[
P_0(y)=y\,,\quad
P_1(y)=1-y^2\,,\quad
P_2(y)=-2y+2y^3\,,\quad
P_3(y)=-2+8y^2-6y^4\,,
\]
and
\[
P_{n+1}(y)=(1-y^2)P_n'(y)\,,\quad\text{for } n\geq 0\,.
\]
These polynomials are called the \emph{derivative polynomials of hyperbolic tangent}. Note that the higher derivatives of $\coth(t)$ are formed in the same pattern as those of $\tanh(t)$. Therefore, $P_n(y)$ are also the derivative polynomials of $\coth(t)$.

The concepts of derivative polynomials for tangent and secant were introduced by Hoffman \cite{Hoff95.DP,Hoff99.DP}, but the study of these polynomials goes back to Knuth and Buckholtz \cite{KnuBu67}, and Krishnamachary and Bhimasena Rao \cite{KriBR24}. More results on these polynomials as well as their hyperbolic analogs can be found in, for example, the papers due to Boyadzhiev \cite{Boy07}, Chu and Wang \cite{ChuWang10.CF}, Cvijovi\'{c} \cite{Cvij09.DP}, Hetyei \cite{Hety08}, and Ma \cite{Ma12.DP}.

Next, from the series expansion of the hyperbolic tangent, we have
\[
\tanh(t)=\sum_{n=1}^{\infty}\frac{2^{2n}(2^{2n}-1)B_{2n}t^{2n-1}}{(2n)!}
=-\sum_{n=1}^{\infty}G_{2n}\frac{(2t)^{2n-1}}{(2n)!}
    =1-\sum_{k=0}^{\infty}\frac{G_{k+1}}{k+1}\frac{(2t)^k}{k!}\,,
\]
which further gives
\begin{equation}\label{Pn}
P_n(\tanh(t))=-\sum_{k=0}^{\infty}\frac{G_{k+n+1}}{k+n+1}\frac{2^{k+n}t^k}{k!}+\de_{n,0}\,,
\end{equation}
where $\de_{n,k}$ is the Kronecker delta. Chu and Wang \cite[Section 2.4]{ChuWang10.CF} used the method of linearization to establish the expansion
\[
P_m(y)P_n(y)=-\rho_{m,n}^{(0)}P_{m+n+1}(y)
    -\sum_{k=1}^{[(m+n)/2]}\rho_{m,n}^{(k)}\frac{2^{2k}B_{2k}}{2k}P_{m+n+1-2k}(y)
    +\chi(m=n=0)\,,
\]
where $\chi$ is used in place of the Iverson bracket, and defined by
\[
\chi(\text{true})=1\quad\text{and}\quad\chi(\text{false})=0\,,
\]
and
\begin{equation}\label{rho.mnk}
\rho_{m,n}^{(k)}=\left\{
    \begin{array}{cc}
        \ds(-1)^m\binom{n}{m+n+1-2k}+(-1)^n\binom{m}{m+n+1-2k}\,,&k\geq 1\,,\\[2ex]
        \ds\frac{m!n!}{(m+n+1)!}\,,&k=0\,,
    \end{array}
\right.
\end{equation}
with $m,n\geq0$. Then
\begin{align}
&[t^l]P_m(\tanh(t))P_n(\tanh(t))\nonumber\\
&\quad=\frac{2^{l+m+n+1}}{l!}\bibb{
    \rho_{m,n}^{(0)}\frac{G_{l+m+n+2}}{l+m+n+2}
    +\sum_{k=1}^{[(m+n)/2]}\rho_{m,n}^{(k)}\frac{B_{2k}G_{l+m+n-2k+2}}
        {(2k)(l+m+n-2k+2)}}\nonumber\\
&\quad\quad+\chi(l=m=n=0)\,,\label{PmPn}
\end{align}
and the following theorem can be established.

\begin{theorem}\label{Th.id.eGG.eBG}
For integers $n,\al,\ga\geq 0$ and $\de,\vep=0,1$, there holds the convolution identity
\begin{align*}
&\sum_{k=0}^n\binom{2n+2-\de-\vep}{2k+1-\de}
    \frac{G_{2k+2\al+2}}{k+\al+1}\frac{G_{2n-2k+2\ga+2}}{n-k+\ga+1}\\
&\quad=4\rho_{2\al+\de,2\ga+\vep}^{(0)}\frac{G_{2n+2\al+2\ga+4}}{n+\al+\ga+2}
    +2\sum_{k=1}^{\al+\ga+[(\de+\vep)/2]}\rho_{2\al+\de,2\ga+\vep}^{(k)}
    \frac{B_{2k}}{k}\frac{G_{2n-2k+2\al+2\ga+4}}{n-k+\al+\ga+2}\,,
\end{align*}
where $\rho_{m,n}^{(k)}$ is defined in (\ref{rho.mnk}).
\end{theorem}

\begin{proof}
Using Eq. (\ref{Pn}), and doing some elementary reduction, we have
\[
P_{2\al+\de}(\tanh(t))=-\sum_{k=0}^{\infty}\frac{G_{2k+2\al+2}}{k+\al+1}
    \frac{2^{2k+2\al}t^{2k+1-\de}}{(2k+1-\de)!}\,,
\]
for $\al\geq0$ and $\de=0,1$. Applying Eq. (\ref{PmPn}) to
$[t^{2n+2-\de-\vep}]P_{2\al+\de}(\tanh(t))P_{2\ga+\vep}(\tanh(t))$,
and considering the above expansion as well as the fact
\[
\chi(2n+2-\de-\vep=2\al+\de=2\ga+\vep=0)=0\,,
\]
we obtain the desired convolution formula.
\end{proof}

Finally, the identity (\ref{id.BG.GG}) can be verified by combining Theorem \ref{Th.id.eGG.eBG} with Eq. (\ref{PmPn}).

\begin{theorem}\label{Th.id.BG.GG}
For integers $n\geq 0$ and $q\geq 2$, there holds the convolution identity
\[
\sum_{i=0}^{q-1}\binom{q-1}{i}\frac{B_{q+i}G_{2n+q-i}}{(q+i)(2n+q-i)}
    +\frac{1}{4}\sum_{i=0}^{2n}(-1)^i\binom{2n}{i}
        \frac{G_{q+i}G_{2n+q-i}}{(q+i)(2n+q-i)}
=\frac{(-1)^q}{q\binom{2q}{q}}\frac{G_{2n+2q}}{2n+2q}\,.
\]
\end{theorem}

\begin{proof}
Let us prove Theorem \ref{Th.id.BG.GG} according to the parity of the parameter $q$. In Theorem \ref{Th.id.eGG.eBG}, setting $\al=\ga=l$ with $l\geq0$ and $\de=\vep=1$, we have
\begin{align*}
&\sum_{k=0}^n\binom{2n}{2k}\frac{G_{2k+2l+2}G_{2n-2k+2l+2}}{(k+l+1)(n-k+l+1)}\\
&\quad=\frac{4}{(4l+3)\binom{4l+2}{2l+1}}\frac{G_{2n+4l+4}}{n+2l+2}
    -4\sum_{k=1}^{2l+1}\binom{2l+1}{4l+3-2k}\frac{B_{2k}G_{2n-2k+4l+4}}{k(n-k+2l+2)}\\
&\quad=\frac{4}{(l+1)\binom{4l+4}{2l+2}}\frac{G_{2n+4l+4}}{n+2l+2}
    -4\sum_{k=0}^{l}\binom{2l+1}{2k}\frac{B_{2k+2l+2}G_{2n-2k+2l+2}}{(k+l+1)(n-k+l+1)}\,,
\end{align*}
for $n,l\geq0$, which is just the $q=2l+2$ case of Theorem \ref{Th.id.BG.GG}. Similarly, in Theorem \ref{Th.id.eGG.eBG}, let $\al=\ga=l+1$ with $l\geq0$ and $\de=\vep=0$, and do some transformation. Then we can show that when $q=2l+3$ and $n\geq 1$, Theorem \ref{Th.id.BG.GG} is still true. Hence, it suffices to show that when $q=2l+3$ and $n=0$, the result holds. In fact, according to (\ref{Pn}), we have
\[
[t^0]P_n(\tanh(t))=-\frac{2^nG_{n+1}}{n+1}+\de_{n,0}\,.
\]
Then by setting $m=n=2l+2$, for $l\geq 0$, we obtain from (\ref{PmPn}) that
\[
[t^0]P_{2l+2}(\tanh(t))^2=0
=2^{4l+5}\bibb{\rho_{2l+2,2l+2}^{(0)}\frac{G_{4l+6}}{4l+6}
    +\sum_{k=1}^{2l+2}\rho_{2l+2,2l+2}^{(k)}\frac{B_{2k}G_{4l-2k+6}}{(2k)(4l-2k+6)}}\,.
\]
This further gives
\begin{equation}
\sum_{k=0}^l\binom{2l+2}{2k+1}\frac{B_{2l+2k+4}G_{2l-2k+2}}{(l+k+2)(l-k+1)}
    =-\frac{2}{(2l+3)\binom{4l+6}{2l+3}}\frac{G_{4l+6}}{2l+3}\,,
\end{equation}
which is just the $q=2l+3$ and $n=0$ case of the desired result.
\end{proof}

\begin{example}
In Theorem \ref{Th.id.BG.GG}, setting further $n=0,1$ gives
\begin{align*}
&\sum_{i=0}^{q-1}\binom{q-1}{i}\frac{B_{q+i}G_{q-i}}{(q+i)(q-i)}
    =\frac{(-1)^q}{q\binom{2q}{q}}\frac{G_{2q}}{2q}
    -\frac{G_q^2}{4q^2}\,,\\
&\sum_{i=0}^{q-1}\binom{q-1}{i}\frac{B_{q+i}G_{q+2-i}}{(q+i)(q+2-i)}
    =\frac{(-1)^q}{q\binom{2q}{q}}\frac{G_{2q+2}}{2q+2}
    -\frac{G_qG_{q+2}}{2q(q+2)}+\frac{G_{q+1}^2}{2(q+1)^2}\,,
\end{align*}
for $q\geq 2$; while setting $q=3,4$ yields
\begin{align*}
&\sum_{i=0}^{2n}(-1)^i\binom{2n}{i}\frac{G_{i+3}G_{2n+3-i}}{(i+3)(2n+3-i)}
    =-\frac{G_{2n+6}}{15(2n+6)}+\frac{G_{2n+2}}{15(2n+2)}\,,\\
&\sum_{i=0}^{2n}(-1)^i\binom{2n}{i}\frac{G_{i+4}G_{2n+4-i}}{(i+4)(2n+4-i)}
    =\frac{G_{2n+8}}{70(2n+8)}+\frac{G_{2n+4}}{30(2n+4)}-\frac{G_{2n+2}}{21(2n+2)}\,,
\end{align*}
for $n\geq 0$. More special cases can be obtained from Theorems \ref{Th.id.eGG.eBG} and \ref{Th.id.BG.GG} by specifying the parameters.\hfill\qedsymbol
\end{example}

For various other convolution identities on the Bernoulli numbers (polynomials), Euler numbers (polynomials) and Genocchi numbers, the readers are referred to, for example, the works of Agoh and Dilcher \cite{AgDil07.CI}, Chu and Wang \cite{ChuWang10.CF}, Gessel \cite{Gess05}, Pan and Sun \cite{PanSun06}, and some further generalizations of their results.


\section*{Acknowledgments}
\addcontentsline{toc}{section}{Acknowledgments}

The first author Weiping Wang is supported by the National Natural Science Foundation of China (under Grant 11671360). The corresponding author Ce Xu is supported by the National Natural Science Foundation of China (under Grant 12101008) and the Scientific Research Foundation for Scholars of Anhui Normal University.


\end{document}